\title{ On the Shinbrot's criteria for energy equality to Newtonian fluids: A simplified proof, and an extension of the range of application.}
\author{Hugo Beir\~{a}o da Veiga$^{1,}$ \footnote{Partially supported  by FCT (Portugal) under the project: UIDB/MAT/04561/2020.}
\qquad Jiaqi Yang$^{2,}$\footnote{Hugo Beir\~{a}o da Veiga (\texttt{hbeiraodaveiga@gmail.com}) and Jiaqi Yang (\texttt{yjqmath@nwpu.edu.cn})}}
\date{
\small $^1$ Department of Mathematics, Pisa University, Pisa, Italy\\
\small $^2$ School of Mathematics and Statistics,
Northwestern Polytechnical University,
Xi'an, 710129, China}
\documentclass[12pt]{article}
\usepackage{amsfonts}
\usepackage{mathrsfs}
\usepackage{amsmath}
\usepackage{amssymb,extarrows}
\usepackage{multicol}
\usepackage{float}
\usepackage{makeidx}
\usepackage{layout}
\usepackage{array}
\usepackage{a4wide}
\usepackage{boxedminipage}
\usepackage{hyperref}
\usepackage{latexsym}
\usepackage{color}
\usepackage{dsfont}
\usepackage{graphicx}
\usepackage{ulem}
\usepackage{amsthm}

\newtheorem{theorem}{Theorem}[section]

\theoremstyle{remark}

\theoremstyle{definition}

\numberwithin{equation}{section}

\newcommand{\f}{\frac}
\newcommand{\n}{\nabla}

\newcommand{\ed}{\end{document}}

\newcommand{\na}{{\nabla}}
\newcommand{\pa}{{\partial}}

\newcommand{\Om}{{\Omega}}

\begin{document}
\maketitle
\begin{abstract}
We show that the classical Shinbrot's criteria to guarantee that a Leray-Hopf solution satisfies the energy equality follows trivially from the $\,L^4( (0\,,T)\times\Om))$ Lions-Prodi particular case. Moreover we extend Shinbrot's result to space coefficients $ r \in (3,\,4)\,.$ In this last case our condition coincides with Shinbrot condition for $r=4$, but for $r<4$ it is more restrictive than the classical one, $ 2/p + 2/r = 1\,.$ It looks significant that in correspondence to the extreme values $r=3$ and $r=\infty$, and just for these two values, the conditions become respectively $u \in L^\infty(L^3)$ and $u \in L^2(L^\infty)$, which imply regularity
by appealing to classical Ladyzhenskaya-Prodi-Serrin (L-P-S) type conditions. However, for values $r\in (3,\infty)$ the L-P-S condition does not apply, even for the more demanding case $\,3<r<4\,.$ The proofs are quite trivial, by appealing to interpolation, with $L^\infty(L^2)$ in the first case and with $L^2(L^6)$ in the second case. The central position of this old classical problem in Fluid-Mechanics, together with the simplicity of the proofs (in particular the novelty of the second result) looks at least curious. This may be considered a merit of this very short note.\par%
\end{abstract}


\vspace{0.2cm}

\vspace{0.2cm}

\section{Introduction and results}
We are interested on the energy equality for weak (Leray-Hopf) solutions to Newtonian incompressible fluids
\begin{equation}\label{eq:NS}
\begin{cases}
\pa_t u+\,u\cdot\nabla\,u-\Delta u+\nabla \pi=0\,,\quad &\textrm{in $\Omega\times(0,T)$}\,,\\
\nabla \cdot\,u=0\,,\quad& \textrm{in $\Omega\times(0,T)$}\,,\\
u=0\,,\quad& \textrm{on $\pa\Omega\times(0,T)$}\,,\\
u(\cdot,\,0)=\,u_0\,, \quad &\textrm{in $\Omega$}\,,
\end{cases}
\end{equation}
where $\,\Om\,$ is a bounded, smooth domain. The energy equality reads
\[
\int_{\Omega}|u(t_0)|^2dx+2\int_0^{t_0}\int_{\Omega}|\na u(\tau)|^2\,dxd\tau=\int_{\Omega}|u_0|^2dx\,,\quad \textrm{for any} \quad t_0\in[0,T).
\]

\vspace{0.2cm}

In reference \cite{shinbrot}  M.Shinbrot shows that if a weak Leray-Hopf solution $\,u\,$ to the Navier-Stokes equations \eqref{eq:NS} satisfies
\begin{equation}\label{shin}
u \in\,L^p (0,\,T; L^r(\Omega))\,,
\end{equation}
where
\begin{equation}\label{shonas}
\frac 2p + \frac 2r =\,1\,,\quad  \textrm{and} \quad r\geq\,4,
\end{equation}
then $\,u\,$ satisfies the energy equality. This result is a generalization of previous results due to G.Prodi \cite{Prodi}  and J.L.Lions \cite{lions}, where these authors proved the above result for $p=\,r=\,4\,.$\par%
For convenience we write the condition \eqref{shin}, \eqref{shonas} in the equivalent form
\begin{equation}\label{shin1}
u \in\,L^{\f{2r}{r-2}}(0,\,T; L^r(\Omega))\,,\quad r\geq 4\,.
\end{equation}
Recall that Leray-Hopf solutions verify $\,u\in L^\infty(0,T; L^2(\Om)) \cap L^2(0,T; W^{1,2}(\Om))\,.$ In particular $\,u\in L^2(0,T; L^6(\Om))\,.$
Below we prove the Theorem \ref{theo-main}. It is worth noting that item (i) is the well know Shinbrot result. But our proof, obtained by a trivial reduction to Lions-Prodi result, is simpler and more significant. Item (ii) looks completely new. We refer the reader to the remarks made in the abstract.
\begin{theorem}\label{theo-main}
Let $u_0 \in H $  and let $u$ be a Leray-Hopf weak solution of the Navier-Stokes equations \eqref{eq:NS}, where $H$ is the completion of $\mathcal{V}=\{\phi \in C^{\infty}_0(\Om):\,\n\cdot \phi=\,0\,\}$ in $L^2(\Omega)$. Assume that $u$ satisfies one of the two conditions below:
\begin{equation}\label{pqcases}
\begin{cases}
(i)\quad  u \in\,L^{\f{2r}{r-2}}(0,\,T; L^r(\Omega))\,,\quad r\in [4,\,\infty]\,.\\
(ii)\quad  u \in\, L^{\frac{r}{r-3}}(0,\,T; L^r(\Om))\,, \quad  r\in [3,\,4]\,.
\end{cases}
\end{equation}
Then $u$ satisfies the energy equality
\begin{equation}\label{eneq}
\| u(t_0)\|_2^2 \,+\, 2\,\int_{0}^{t_0} \,\|\n u(\tau)\|_2^2 \,d\tau=\,\| u_0\|_2^2\,,
\end{equation}
for any $t_0\in [0,T)$.
\end{theorem}
\begin{proof}
Proof of item (i):\par%
We appeal to $\,u\in L^\infty(L^2).$ By interpolation we show that
$$
\|u\|_4 \leq\,\|u\|_2^{\f{r-4}{2(r-2)}} \|u\|_r^{\f{r}{2(r-2)}}\,, \quad \textrm{for} \quad r\geq 4\,.
$$
By integration in $(0,\,T)$ one easily proves that the estimate
$$
\|u\|^4_{L^4(L^4)}\leq \|u\|^{\f{2(r-4)}{r-2}}_{L^\infty(L^2)} \|u\|^{\f{2r}{r-2}}_{L^{\f{2r}{r-2}}(L^r)}
$$
holds. The energy inequality follows from Lions-Prodi result.%

\vspace{0.2cm}

Proof of item (ii):\par%
Now we appeal to $\,u\in L^2(L^6).$ Again by interpolation we show that
$$
\|u\|_4 \leq\,\|u\|_6^{\f{3(4-r)}{2(6-r)}} \|u\|_r^{\f{r}{2(6-r)}}\,, \quad \textrm{for} \quad 1\leq r\leq 4.
$$
Hence
$$
\int_{0}^{T} \|u\|^4_4 \, dt  \leq\,\int_{0}^{T} \|u\|_6^{\f{6(4-r)}{6-r}}\|u\|_r^{\f{2r}{6-r}} \,dt\,,
$$
for $\,r\,$ as above. Next we apply H\H older inequality with exponents $\,s=\f{6-r}{3(4-r)}$ and $\,s'=\f{6-r}{2(r-3)}\,,$ for $3<r\leq4\,.$
This leads to
$$
\|u\|^4_{L^4(L^4)} \leq\,\|u\|^{\f{6(4-r)}{6-r}}_{L^2(L^6)} \, \|u\|^{\f{2r}{6-r}}_{L^{\f{r}{r-3}}(L^r)}\,.
$$
The thesis of item (ii) follows.
\end{proof}
Let's comment on the above results. First of all note that the results in item (i) and (ii) glue perfectly for the common value $r=4$.\par%
Furthermore, the Shinbrot number
\begin{equation}\label{shonum}
\theta(p,\,r):=\frac 2p + \frac 2r\,,
\end{equation}
which is equal to $1$ in item (i), is given in item (ii) by $\,2-\f4r\,.$ This number decreases with $r\,,$ and is less than the classical value $\,1\,,$ for $r<4$. Note that, at least formally, larger is the Shinbrot number better is the result.\par
Another main remark is the following. It is well known that if a Leray-Hopf solution $u$ satisfies the so called Ladyzhenskaya, Prodi, Serrin (LPS) condition, namely
\begin{equation}\label{LPS}
u \in\,L^p (0,\,T; L^r(\Omega))\,,
\end{equation}
where
\begin{equation}\label{BOH}
\frac 2p + \frac 3r =\,1\,,\quad  \textrm{for} \quad r \in [3,\,\infty],
\end{equation}
then $\,u\,$ is smooth. In this case the energy equality follows from the smoothness. Hence it is worth noting that the conditions imposed in the above two items do not fail inside the range of the LPS condition, for $\,3<r<\infty.$ Furthermore, and this fact is particularly significant, the extremal values $r=\infty$ in item (i) and $r=3$ in item (ii) lead respectively to $u\in L^2(L^\infty)$ and $u\in L^\infty(L^3)\,$, which are two (quite particular) cases of LPS condition.
\section{Comparison with previous results}
In references \cite{BV-JY} and \cite{BC} the authors proved the energy equality under conditions of type
\begin{equation}\label{shiq}
u \in\,L^p (0,\,T; W^{1,\,q}(\Omega))\,,
\end{equation}
instead of type \eqref{shin1}. However the above Shinbrot's number $\theta(p,r)$ was extended to
this new situation  by appealing to the sharp Sobolev's embedding theorem
\begin{equation}\label{sobolef}
W^{1,\,q}(\Omega) \subset \,L^r(\Omega)\,,
\end{equation}
where
\begin{equation}\label{sobolev}
\f1q=\,\f1r + \f13\,,
\end{equation}
and $ q<3\,.$ In the above references the authors compare the two type of results,  \eqref{shin1} and  \eqref{shiq}, at the light of Shinbrot's number, assumed as a measure of strength. This comparison was extended to other results, obtained by different authors. We refer the reader to the above references. In the sequel we merely extend to the new range $\,r\in(3,4]\,$  the comparison between the results stated in item (ii)  and the results in \eqref{shin1} and  \eqref{shiq}, the natural complement to the similar comparison done in these references with respect to the classic range  $\,r\in[4,\infty)\,.$ Note that, by \eqref{sobolev}, one has
\begin{equation}\label{equiv}
\f32<q \leq \f{12}{7}\,,       \quad  3<r\leq 4\,.%
\end{equation}
Hence we will restrict the conditions of type \eqref{shiq} to the above $q-$range. Inside this range, in \cite{BC} theorem 1 item (i) the authors proved energy equality under assumption \eqref{shiq}, where $ p=\f{q}{2q-3}.$ By appealing to  \eqref{sobolev}  one gets $\,p=\,\f{r}{r-3}\,. $ This immediately shows that,  curiously, the Shinbrot's number is given by $2-\,\f{4}{r}$, exactly as shown for the solutions  in the above theorem \ref{theo-main}, item (ii).\par%
A similar calculation applied to the Theorem 4.4 item (i) in  reference \cite{BV-JY} leads to the better, constant Shinbrot's number $\f{10}{9}$, still obtained for $\,q\leq\,\f95\,$ ($r\leq\, 4+\,\f12)\,.$


\begin{thebibliography}{99}
%
\bibitem{BV-JY}
Beir\~{a}o da Veiga,~H. and Yang,~J.: On the energy equality for solutions to Newtonian and non-Newtonian fluids, Nonlinear Analysis, \textbf{185} (2019), 388--402.
\bibitem{BC}
Berselli,~L. C. and Chiodaroli,~E.: On the energy equality for the 3D Navier-Stokes equations, Nonlinear Analysis, \textbf{192}, (2020), 111704.
\bibitem{lions}
\newblock Lions,~J.L.:
\newblock {Sur l'existence de solutions des \'equations de Navier-Stokes.}
\newblock  C. R. Acad. Sci. Paris, \textbf{248}, 2847--2849 (1959)%
%
\bibitem{Prodi}
Prodi,~G.: Un teorema di unicit\`a per le equazioni di Navier-Stokes, Ann. Mat. Pura Appl., \textbf{48}, 173--182 (1959)
%
\bibitem{shinbrot}
\newblock Shinbrot,~M.:
\newblock {The energy equation for the Navier-Stokes system.} SIAM J. Math.Anal., \textbf{5}, 948-954, (1974).
%
\end{thebibliography}
\end{document}